\documentclass[11pt,reqno]{amsart}
\usepackage{graphicx,amsmath,amsthm,verbatim,tikz,enumitem,hyperref, amssymb, txfonts}
\usepackage{tikz-cd}
\usetikzlibrary{arrows,decorations.pathmorphing}
\usepackage[margin = 1.25 in]{geometry}
\hypersetup{colorlinks,linkcolor={red},citecolor={olive},urlcolor={red}}
\usepackage{mathtools}
\mathtoolsset{showonlyrefs}

\newcommand{\Z}{\mathbb{Z}}

\newcommand{\E}{\mathbb{E}}
\renewcommand{\P}{\mathbb{P}}

\newcommand{\f}{\frac}
\newcommand{\ind}[1]{\mathbf 1 \{#1 \}}
\renewcommand{\emptyset}{\varnothing}
\renewcommand{\phi}{\varphi}

\newcommand{\meet}{\text{\:---\:}}

\newcommand{\lrl}{\longleftrightarrow}
\newcommand{\lrlb}{\overset{\B}{\longleftrightarrow}}
\newcommand{\ra}{\rightarrow}
\newcommand{\la}{\leftarrow}
\newcommand{\raw}{\rightsquigarrow}
\newcommand{\law}{\leftsquigarrow}
\newcommand{\ras}{\Rightarrow}
\newcommand{\las}{\Leftarrow}
\newcommand{\dil}[2]{\overleftarrow{D}^{({#2})}_{#1}}
\newcommand{\dir}[2]{\overrightarrow{D}^{({#2})}_{#1}}
\newcommand{\harr}{\overset{\BB}{\longleftrightarrow}}

\newcommand{\cev}[1]{\reflectbox{\ensuremath{\vec{\reflectbox{\ensuremath{#1}}}}}}
\renewcommand{\b}{\bullet}
\renewcommand{\L}{\cev{\b}}
\newcommand{\R}{\vec{\b}}
\newcommand{\B}{\dot{\b}}
\newcommand{\BB}{\hat{\b}}

\usepackage{theoremref}

\newtheorem{theorem}{Theorem}
\newtheorem{lemma}[theorem]{Lemma}
\newtheorem{proposition}[theorem]{Proposition}

\newtheorem*{claim}{Claim}

\newtheorem{remark}[theorem]{Remark}
\theoremstyle{definition}

\title{Four-parameter coalescing ballistic annihilation}

\author[Affeld]{Kimberly Affeld} \email{kaffeld@vassar.edu}
\author[Dean]{Christian Dean} \email{deanmchris@gmail.com}
\author[Junge]{Matthew Junge} \email{Matthew.Junge@baruch.cuny.edu}
\author[Lyu]{Hanbaek Lyu}  \email{hlyu@math.wisc.edu}
\author[Panish]{Connor Panish} \email{wilsonc1@ufl.edu}
\author[Reeves]{Lily Reeves} \email{lreeves@caltech.edu}

\thanks{Affeld, Dean, Junge, Panish, and Reeves were partially supported by NSF grants DMS-2115936 and DMS-2115936. Reeves was partially supported by NSF grant DMS-2303316. Lyu was partially supported by NSF grant DMS-2010035. Part of this research was completed during the 2022 Baruch College Discrete Math REU partially supported by NSF grant DMS-2051026. }

\begin{document}
	\maketitle
	
	\begin{abstract}
		In coalescing ballistic annihilation, infinitely many particles move with fixed velocities across the real line and, upon colliding, either mutually annihilate or generate a new particle. We compute the critical density in symmetric three-velocity systems with four-parameter reaction equations. 
  	\end{abstract}

	\section{Introduction}
	
    Annihilating systems model $A+B \to \varnothing$ chemical reactions  \cite{toussaint1983particle, bramson1991asymptotic,  cabezas2018recurrence, johnson2023particle}. Intriguing critical behavior and rich combinatorial structure led to the study of systems with ballistic rather than diffusive motion \cite{elskens1985annihilation, droz95}. The canonical such process, \emph{ballistic annihilation} (BA), is defined as follows. 
    For each integer $k$, we let $\b_k$ represent the $|k|$th particle to the right or left of the origin ($k>0$ for right and $k<0$ for left) whose initial location is denoted by $x_k \in \mathbb R$. We set $x_0 = 0$ and sample $x_k$ so that the the spacings $x_k - x_{k-1}$ are independently sampled from the same  continuous distribution supported on a subset of $(0,\infty)$. Each particle is independently assigned a velocity according to the same distribution. Particles move at their assigned velocities and mutually annihilate upon colliding. 
    

    The \emph{symmetric three-velocity setting} has received the most attention. Velocity 0 particles, which we will refer to as \emph{blockades}, occur with probability $p$. Velocity $+1$ and $-1$ particles, which we will call \emph{right} and \emph{left arrows}, respectively, each occur with probability $(1-p)/2$. We define the critical initial density of blockades
    \begin{align}
    p_c := \sup\{ p \colon \P_p(\text{no blockades survive}) =1 \}.\label{eq:pcd}
    \end{align}
%
    Most interest has revolved around computing $p_c$ and describing the phase-behavior at and away from criticality. Droz et.\ al and later Krapivsky et.\ al \cite{droz95, krapivsky95} deduced that $p_c=1/4$. Despite some intial progress \cite{ST, bullets, burdinski}, this equality was not rigorously established until a breakthrough from Haslegrave, Sidoravicius, and Tournier \cite{HST}. 

    Inspired by the physics literature \cite{blythe2000stochastic} and a desire to further explore the method from \cite{HST}, Benitez, Junge, Lyu, Redman, and Reeves introduced a (symmetric) coalescing variant in which collisions sometimes generate new particles \cite{benitez2023three}.  We generically denote the three particle types with 0,+1, and $-1$ velocities as $\B, \R,$ and $\L$, respectively.
The notation $[\b \meet \b \ras \Theta, \quad x]$ encodes a reaction rule; upon colliding, the particle $\Theta \in \{\B, \L, \R, \emptyset\}$ is all that remains with probability $x$. Fix parameters $0 \le a, b, \alpha, \beta < 1$ with $a + b \le 1$, $\alpha+ \beta \le 1$. \emph{Four-parameter coalescing ballistic annihilation} (FCBA) has the following reactions:
	\begin{align}
		\R \meet \L \ras 
		\begin{cases} 
			\L, & a/2 \\ 
			\R , & a/2 \\ 
			\B, & b \\ 
			\emptyset , & 1- (a+b)
		\end{cases}, \qquad 
  \B \meet \L \ras 
			\begin{cases} 
				\L , & \alpha \\ 
				\B, & \beta \\ 
				\emptyset , & 1-(\alpha + \beta)
			\end{cases}, \qquad 
   \R \meet \B \ras 
			\begin{cases} 
				\R , & \alpha \\ 
				\B, & \beta \\ 
				\emptyset , & 1-(\alpha + \beta)
			\end{cases}.
	\end{align}
See Figure~\ref{fig:pc} for a visualization.

When the velocity of a newly generated particle matches the velocity of one of the reactants, it is convenient to view it as a continuation of the previous particle. For example, we will view a blockade $\B_k$ that has undergone multiple $\R \meet \B_k \implies \B$ reactions as the same blockade surviving collisions. With this perspective we continue to define $p_c$ as at \eqref{eq:pcd}; ``no blockades survive" means that each blockade initially present in the system will be annihilated after an almost surely finite time.

\begin{figure}
\includegraphics[width =  .9\textwidth]{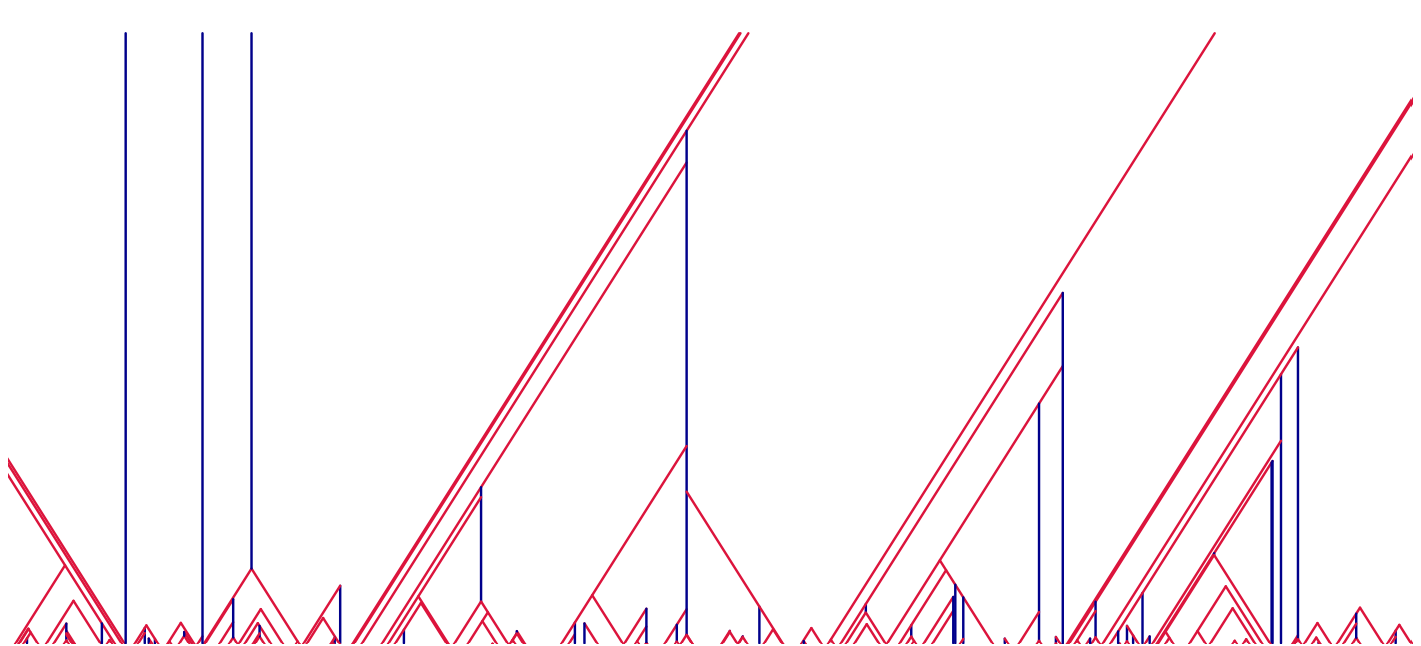}
\caption{A simulation of FCBA with 200 particles. The four reaction parameters are set to $1/3$ and $p=0.15$. \thref{thm:main} tells us that $p_c = 2/13 \approx 0.1538$. Space is the horizontal axis and time is the vertical axis. Lines show particle trajectories. Blue represents blockades and red arrows. The code used to generate this figure can be found at: \href{https://github.com/algerbrex/ballistic\_annihilation}{https://github.com/algerbrex/ballistic\_annihilation}.
} \label{fig:pc}
\end{figure}

The value of $p_c$ was computed in \cite{benitez2023three} for three-parameter systems with $\beta=0$. Junge, Ortiz, Reeves, and Rivera computed $p_c$ for $\beta \geq 0$, but the other parameters were set to zero \cite{junge2023non}. Our main theorem is a formula for $p_c$ that allows for the full four-parameter space. Its complexity illustrates the subtle interplay between the parameters in promoting and impeding blockade survival.  

 	\begin{theorem} \thlabel{thm:main}
		For any FCBA it holds that
		\begin{equation} \label{eq:pc}
			p_c  = \frac{(1-\beta)^2 - b (1-\alpha)}{4-(a+b)(1-\alpha) - \beta (3-\alpha- \beta) - 3\alpha}.
		\end{equation}
	\end{theorem}

As in \cite{junge2023non}, we can implicitly describe the probability the origin is visited by a particle. Let $(0 \la \b)_{(0,\infty)}$ denote the event that the origin is visited by a left arrow in FCBA restricted to only the particles started in $(0,\infty)$. Suppressing the dependence on the four parameters, let $q=q(p):= \P(  (0 \la \b)_{(0,\infty)})$. 

\begin{theorem}\thlabel{thm:q} For any FCBA and $0\leq p \leq 1$, $q$ solves the equation in \thref{prop:qr}. Moreover, $q$ is continuous with $q(p) =1$ for $0\leq p \leq p_c$, and $q$ is strictly decreasing for $p_c \leq p \leq 1$.
\end{theorem}

\subsection{Proof Overview} \label{sec:overview}
Though we invoke a lot of the framework from \cite{benitez2023three, junge2023non}, our result is not simply a stitching together of these two results.
There are three main novelties. First, the result generalizes the formula for $p_c$ to a broader family of reactions. Second, we employ an efficient framework to organize and attack the somewhat involved calculations. Lastly, the proof resolves an asymmetry noted in \cite[Remark 2.6]{benitez2023three} that was only partially addressed in \cite{junge2023non}. We isolate and explain that difficulty now. 

An important event is $A:= \{(0 \la \b)_{(0,\infty)} \wedge (\R_1 \lrl \B)\}$ i.e., the origin is visited in the process restricted to the particles in $(0,\infty)$ when $\R_1$ is a right arrow that mutually annihilates with a blockade. In non-coalescing ballistic annihilation, this event ``looks" like the figure below:
    \begin{center}
        \begin{tikzpicture}[scale = 1.4,
      >=stealth',
      pos=.8,
      photon/.style={decorate,decoration={snake,post length=1mm}}
    ]
        \draw[->] (-.5,-1.5) to node[below]{} (7,-1.5);
    
        \node[inner sep=0pt](a1) at (0,-1.2)
        {{{{$\R_1$}}}};
        
        \node[inner sep=0pt](a1) at (3,-1.2)
        {{{{$\B_k$}}}};
        \draw[->,color=black] (0,-.9) to [bend left = 15] (2.9, -0.9);
        \node[inner sep=0pt](a1) at (-.5,-1.5)
        {\tiny{$|$}};
        \node[inner sep=0pt](a1) at (-.5,-1.8)
        {0};
        \node[inner sep=0pt](a1) at (6.25,-1.2)
        {{{{$\L_n$}}}};
        \draw[->,color=black] (6,-.9) to [bend right = 15] (0, -0.7);
        \draw[|-|] (3,-2) -- (6.25,-2) node [midway,fill=white] {$\cev D_1$};
\draw[|-|] (0,-1.8) -- (3,-1.8) node [midway,fill=white] {$\vec D_1$};    
        \end{tikzpicture}
    \end{center}
Namely, $\R_1$ annihilates with some blockade, say $\B_k$, which is visited later by a left arrow $\L_n$. Since $\B_k$ and all particles in $(x_1,x_k)$ are destroyed, $\L_n$ will visit 0 resulting in the desired event. This requires that the distances $\vec D_1 = x_k - x_1$ and $\cev D_1 = x_n - x_k$ satisfy $\vec D_1 < \cev D_1 < \infty$. A major insight in \cite{HST} was recognizing this structure and noting that $\vec D_1$ and $\cev D_1$ are continuous independent and identically distributed random variables with $\P(\vec D_1 < \infty)=q$. This gives, $$\P( \vec D_1 < \cev D_1 < \infty) = \P(\vec D_1 < \cev D_1 \mid \vec D_1, \vec D_1 < \infty) q^2 = (1/2)q^2.$$ Adding in a factor of $p$ to account for $\B_k$ being a blockade yields $\P(A) = (1/2) p q^2$. 

When $\beta >0$, the event $A$ becomes more complicated because $A$ can occur when multiple arrows visit and are destroyed by $\B_k$ before $\R_1$ and $\L_n$ arrive. For example, in the figure below $\R_1$ is the $i=2$nd right arrow to arrive to $\B_k$ (the pink arrow was destroyed earlier by $\B_k$), and $\L_n$ is the $j=4$th left arrow to arrive (the three orange arrows are destroyed by $\B_k$ before $\R_1$ arrives). We then have $\R_1$ mutually annihilates with $\B_k$ followed by an arrival from $\L_n$ which proceeds to the origin:
     \begin{center}
    \begin{tikzpicture}[scale = 1.4,
      >=stealth',
      pos=.8,
      photon/.style={decorate,decoration={snake,post length=1mm}}
    ]
    \draw[->] (-.5,-1.5) to node[below]{} (7,-1.5);
    \node[inner sep=0pt](a1) at (0,-1.2)
    {{{{$\R_1$}}}};
    \draw[->,color=black] (0,-.9) to [bend left = 15] (2.7, -0.9);
    \node[inner sep=0pt](a1) at (1,-1.2)
    {{{\textcolor{magenta}{$\R$}}}};
    \draw[->,photon,color=black] (1.2,-1.25) to (2.8, -1.25);
    \node[inner sep=0pt](a1) at (3,-1.2)
    {{{{$\B_k$}}}};
        \node[inner sep=0pt](a1) at (4.2,-1.2)
    {{{\textcolor{orange}{$\L$}}}};
    \draw[->,photon,color=black] (4,-1.25) to (3.2, -1.25);
    \node[inner sep=0pt](a1) at (5.2,-1.2)
    {{{\textcolor{orange}{$\L$}}}};
    \draw[->,photon,color=black] (5,-1.25) to (4.3, -1.25);
    \node[inner sep=0pt](a1) at (5.8,-1.2)
    {{{\textcolor{orange}{$\L$}}}};
    \draw[->,photon,color=black] (5.7,-1.25) to (5.3, -1.25);
    \node[inner sep=0pt](a1) at (-.5,-1.5)
    {\tiny{$|$}};
    \node[inner sep=0pt](a1) at (-.5,-1.8)
    {0};
    \node[inner sep=0pt](a1) at (6.25,-1.2)
    {{{$\L_n$}}};
    \draw[->,color=black] (6,-.9) to [bend right = 15] (0, -0.7);
\draw[|-|] (3,-1.8) -- (5.7,-1.8) node [midway,fill=white] {$\cev D_3$};
\draw[|-|] (3,-2.3) -- (6.25,-2.3) node [midway,fill=white] {$\cev D_4$};
\draw[|-|] (0,-2) -- (3,-2) node [midway,fill=white] {$\vec D_2$};    
    \end{tikzpicture}
    \end{center}
    
  The relevant distance comparisons in the event depicted above are $\cev D_3 < \vec D_2 < \cev D_4<\infty$. In general, $A$ includes events like the above figure for all pairs  $i,j\geq 1$. We argue for this formally in \thref{prop:sr}, but the intuition from the formula for $\P(A)$ from \cite{HST} suggests that
  \begin{align}
  \P(A) =p(1-(\alpha + \beta))  \sum_{i,j\geq 1} \beta^{i+j-2} \P( \cev D_{j-1}<\vec D_i < \cev D_{j} < \infty ).\label{eq:asym}
  \end{align}
  
 In reference to a precursor to \eqref{eq:asym}, \cite[Remark 2.6]{benitez2023three} claimed that ``there is no obvious symmetry that allows us to compute the value." The authors of \cite{junge2023non} made partial progress, but only had to compare two, rather than three, distances and their sum contained complement events with the same coefficients.  We evaluate \eqref{eq:asym} in \thref{lem:triple-dist}. 
 A key step is the observation that
$$\P(\cev D_{j-1} <\vec  D_i <\cev  D_{j} < \infty) = \P(\cev D_{j-1} < \vec D_i < \infty \text { and } \cev D_j < \infty) - \P(\cev D_j < \vec D_i < \infty).$$
This lets us rewrite \eqref{eq:asym} in terms of two infinite sums involving comparisons of two distances. After making this substitution into \eqref{eq:asym}, there is still an issue with mismatched exponents in the coefficients. However, the two sums share a common term $\sum_{i,j\geq 1} \beta^{i+j} \P(\vec D_i < \cev D_j< \infty )$ that can be factored out and solved for explicitly (\thref{lem:cap-s}). 

We now provide a sketch for the broader argument and outline of the document organization. In Section~\ref{sec:recursion}, we obtain a recursive expression involving $q$ by partitioning on the velocity of $\b_1$. Deferring the main calculations for later, we use this recursion and the framework from \cite{junge2023non} to prove our main results. 
In Section~\ref{sec:calculations} we derive a formula for a more general version of \eqref{eq:asym} and other important quantities. Having four parameters makes the calculations complex. An invaluable tool for organizing and computing the relevant quantities is a \emph{mass transport principle} at \eqref{eq:mtp}, developed in \cite{junge2022phase}, that uses translation invariance to separate the relevant collisions into independent one-sided events.

\subsection{Further Questions}
We focus on the three-velocity setting. Two natural directions for the study of coalescing systems are to extend to new reaction types, such as $\R \meet \B \implies \L$, or to find $p_c$ for an asymmetric system. The other natural reactions to include do not seem amenable to the recursion technique from \cite{HST}. As for asymmetric reactions, it was demonstrated in \cite{junge2023non} that asymmetric ballistic annihilation systems are significantly more difficult to analyze. For example, finding the value of $p_c$ remains open for \emph{any} asymmetric three-velocity system. 

One direction where perhaps some progress could be made is proving a universality result for the index of the first arriving particle as in \cite{HST, padro2022arrivals} for FCBA. However, success is not assured. The calculations in \cite{padro2022arrivals} involved finitely many particles. Fixing the number of particles poses a serious difficulty when making various distance comparisons, because it introduces an extra constraint to already delicate calculations. 

A straightforward generalization of this work is to introduce two extra parameters dictating reaction probabilities for blockades $\hat \b$ not initially present in the system. For example, we could specify $[\R \meet \hat \b \implies \R, \quad \psi]$ and $[\R \meet \hat \b \implies \hat \b, \quad \phi]$ for parameters  $\psi$ and $\phi$ not necessarily equal to $\alpha$ and $\beta$, respectively. Even more complexity could be introduced if we allowed arrow-blockade collisions to generate multiple types of blockades each with their own reaction probabilities. We worked out the formula for $p_c$ in this six-parameter setting and encountered no major surprises or difficulties. However, the calculations are dense even with four parameters, so we opted for simpler presentation. 
 
\subsection{Notation}
	 Let $c = 1- (a+ b)$ and $\xi = 1 - (\alpha + \beta)$.  
  The recursion argument in the next section involves careful partitioning. We now define and notate a variety of collision and visit events. In what follows, the reflection of the events are defined similarly when applicable. It is convenient to denote the events for different velocity assignments as:
\begin{align*}
    \B_k &= \{\b_k \text{ is a blockade}\}, \quad   
    \R_k = \{\b_k \text{ is a right arrow}\},\quad  \text{ and }
    \L_k = \{\b_k \text{ is a left arrow}\}.
\end{align*}
We now define the various collision events. We denote blockades that came into existence after time $0$ by $\hat \b$ generically and $\hat \b_{j,k}$ if the new blockade was formed by $\R_j \meet \L_k$.
	\begin{align*}
		(\R_j \lrl \L_k) &:= \{\text{$\R_j$ and $\L_k$ mutually annihilate}\} \\
		(\R_j \ra \L_k) &:= \{\text{$\L_k$ is annihilated by $\R_j$, $\R_j$ survives}\} \\
		(\R_j \harr \L_k) &:= \{\text{$\R_j$ and $\L_k$ mutually annihilate and generate a blockade $\BB_{j,k}$} \}\\
		(\B_j \lrl \L_k) &:= \{\text{$\B_j$ and $\L_k$ mutually annihilate}\} \\
		(\B_j \las \L_k) &:= \{\text{$\B_j$ is annihilated by $\L_k$, $\L_k$ survives}\} \\
		(\B_j \law \L_k) &:= \{\text{$\L_k$ is annihilated and $\B_j$ survives}\} \\
		(\B_j \ra \L_k) &:= (\B_j \lrl \L_k) \cup (\B_j \law \L_k).
	\end{align*}
	We call events with `$\law$' \emph{weak} collisions and events with `$\las$' \emph{strong} collisions. 
	When $\beta>0$, blockades can survive multiple weak collisions. There are several events involving blockades:
	\begin{align*}
		(\B_j \overset{i}{\lrl} \L_k) &:= \{\text{$\B_j$ is annihilated by a mutual annihilation with $\L_k$} \\
		&\qquad \qquad  \text{after $i-1$ weak collisions from the right}\} \\
		(\B_j \overset{i}{\lrl} \L) &:= \cup_{k>j}\ (\B_j \overset{i}{\lrl} \L_k) \\
		(\B_j \overset{i}{\las} \L_k) &:= \{\text{$\B_j$ is annihilated by a strong collision with $\L_k$} \\
		&\qquad  \qquad  \text{after $i-1$ weak collisions from the right}\} \\
		(\B_j \overset{i}{\las} \L) &:= \cup_{k>j}\ (\B_j \overset{i}{\las} \L_k) \\
		(\B_j \overset{i}{\law} \L_k) &:= \{\text{$\L_k$ is the $i$th to weakly collide with $\B_j$ from the right}\} \\
		(\B_j \overset{i}{\law} \L) &:= \{\text{$\B_j$ is visited by \emph{exactly} $i$ weak collisions from the right}\}.
	\end{align*}

      We use a subscript $I \subset \mathbb R$ on collision events to restrict to the system that only contains particles that started in $I$. If no subscript is present, the default is the one-sided process restricted to particles in $(0,\infty)$.
    
	Visits are defined as $(u \la \L_k) := \{\text{$u$ is visited by $\L_k$ from the right}\}.$ \emph{Mutual}, \emph{strong}, and \emph{weak} visits are defined as conditional events:
	\begin{align*}
		(u \overset{i}{\lrl} \L_k)_{(u, \infty)} &:= (\B_u \overset{i}{\lrl} \L_k)_{[u, \infty)} \mid \B_u \\
		(u \overset{i}{\las} \L_k)_{(u, \infty)} &:= (\B_u \overset{i}{\las} \L_k)_{[u, \infty)} \mid \B_u \\
		(u \overset{i}{\law} \L_k)_{(u, \infty)} &:= (\B_u \overset{i}{\law} \L_k)_{[u, \infty)} \mid \B_u.
	\end{align*}
	In an abuse of notation, we use $\B_u$ here to denote a blockade at $u$. We consider an arrow colliding with a blockade as also visiting the site of the blockade, e.g. $\{\B_j \meet \L_k\} \subset \{x_j \la \L_k\}$.


	

	\section{Recursion and Proofs of Main Results} \label{sec:recursion}
 The main work of this article is deriving a recursion for $q:= \P( (0 \la \b)_{(0,\infty)})$. We give the details now, while deferring some of the calculations to the next section. 
  \begin{proposition} \thlabel{prop:qr}
  It holds for any $0\leq p\leq 1$ that $0 = g(p,q(p))$ with
\begin{align}
g(u,v):= -q+ \f{f_1(u,v)+ f_2(u,v)}{f_3(u,v)}
\label{eq:g}
\end{align}
for the functions (introduced so that $g$ can fit on one line):
\begin{align}
f_1(u,v) &:= u v \left(-v \left(-\alpha  a+a+\alpha +\beta  (\alpha +\beta ) \left(b v^2-1\right)-b \beta  v (v+2)+b+\beta -1\right)-2 (\alpha +\beta -1)\right)\\
f_2(u,v) &:= (\beta  v-1)^2 (v (a+b (v-1) v-2)-1)+u\\
f_3(u,v) &:= (\beta  v-1)^2 (a+b (v-2) v-2).
\end{align}    
  \end{proposition}

\begin{proof}
    Let $B = ( 0 \la \b)_{(0,\infty)}$ so that $q := \P(B)$. The starting point is the partition
 \begin{align}
 q = \P(B \wedge \L_1)  + \P(B \wedge \B_1) + \P(B \wedge \R_1).\label{eq:qr}
 \end{align}
 We will go step by step to derive a formula for each term.
 
 First off, $\P(B \wedge \L_1) = \P(\L_1) = (1-p)/2$, since if the first particle is left moving, $0$ will be visited. The second term is a bit more involved, but fairly straightforward. We simply need to account for the number of weak visits to $\B_1$ and how it is finally destroyed. It is proven in \thref{lem:main-rec} that 
 $$\P(B \wedge \B_1)
			= \frac{\alpha pq + \xi pq^2}{1 - \beta q}.$$
 
The last term is the most complicated. Define
	\begin{align} \label{eq:sr}
		s & := \P(B \wedge (\R_1 \la \B)) + \P((0\la \b) \wedge (\R_1 \la \BB)), \\
		r & := \P(B^c \wedge (\R_1 \la \B)) + \P((0 \not \la \b) \wedge (\R_1 \la \BB)).
	\end{align}
 \cite[Lemma 2.3 (2.4)]{benitez2023three} proved that 
 \begin{align}
     \P(B \wedge \R_1)
			&= \left(q + \frac{a/2}{c} + \frac{b}{c} \frac{\P((0 \la \b) \wedge \B_1)}{p}\right) \frac{\frac{1-p}{2} - s - r}{1 + \frac{a/2}{c} + \frac{b}{c}} +s.
 \end{align}
 The argument does not involve $\beta$ and is identical for our setting (see \thref{rem:c} for what to do if $c=0$).
 So, we have gone from \eqref{eq:qr} to the expanded recursion
 \begin{align}
 q = \f{1-p}{2}  + \frac{\alpha pq + \xi pq^2}{1 - \beta q}  +  \left(q + \frac{a/2}{c} + \frac{b}{c} \frac{\P((0 \la \b) \wedge \B_1)}{p}\right) \frac{\frac{1-p}{2} - s - r}{1 + \frac{a/2}{c} + \frac{b}{c}} + s.\label{eq:qr2}
 \end{align}
All that is missing are formulas for $s$ and $r$. Let $\hat{p} := \P((\R_1 \lrlb \L)_{[x_1, \infty)})$. We prove in \thref{prop:sr} that 
		\begin{align}\label{eq:sr1}
			s &= \frac{(p+\hat p)(1-\alpha)(1-\beta)q^2}{2 (1-\beta q)^2}
		\qquad \text{ and } \qquad 
			r = \frac{(p+\hat p) \alpha q(1-q)}{(1-\beta q)^2}. 
		\end{align}
  Substituting the values of $s$ and $r$ at \eqref{eq:sr1} into \eqref{eq:qr2} and subtracting $q$ from both sides gives the claimed recursion.
\end{proof}
  
\begin{proof}[Proof of \thref{thm:main} and \thref{thm:q}]
    The proofs are nearly identical to the analogous arguments given in \cite{junge2023non}. We first observe that continuity of $q$ can be proven following the argument in \cite[Section 3]{junge2023non} (which, itself, closely followed \cite[Section 3]{junge2022phase}). The only difference is the superadditivty property in ``Step 1" requires more generality. Fortunately, the proof of this Step 1 refers to \cite[Lemma 15]{benitez2023three}, which is stated and proven at the level of generality of FCBA ($\beta >0$). So, ``Step 1'', and all subsequent steps still hold.
    Next, \thref{prop:qr} ensures that $(p,q(p))$ solves $0=g(p,q(p))$ with $g$ defined at \eqref{eq:g}. The argument in \cite[Section 4]{junge2023non} involves analyzing a similar implicit equation. That section provides a general framework for proving that $p_c$ is the unique solution to the linear equation $g(u,1)=0$, and that $q$ transitions from solving the equation $1-v=0$ for $p \leq p_c$ to a strictly decreasing root of $g(u,v)/(1-v)$ for $p_c < p \leq 1$. We omit the details since it the argument is very similar.
\end{proof}



\section{Calculations} \label{sec:calculations}
We first compute the comparisons alluded to in Section~\ref{sec:overview}. Next, we give two useful tools that we subsequently use to derive the quantities invoked in the proof of \thref{prop:qr}. 

\subsection{Arrival comparisons} 

    Let us define $\dir{i}{u}$ and $\dil{i}{u}$ as the times the location $u$ is visited for the $i$th time in FCBA restricted to $(-\infty, u)$ and $(u, \infty)$, respectively. Since arrows move with speed $1$, $\dir{i}{u}$ and $\dil{i}{u}$ also represent distances. For convention, we set $\dir{0}{u} = \dil{0}{u} = 0$. By symmetry, $\dir{i}{u}$ and $\dil{i}{u}$ are i.i.d. By ergodicity, $\dir{i}{u_1}$ and $\dil{i}{u_2}$ are identically distributed whenever $u_1 \le u_2$, under which they are also independent.
	
	Visits to $u$ enjoy the renewal property, i.e., the $(i+1)$th particle to visit $u$ from the right is the first to visit the initial location of the $i$th particle to visit $u$ from the right. This, along with translation invariance, implies 
	\begin{align*}
		\dil{i}{u} = X_1 + \cdots + X_i,
	\end{align*}
	where each $X_k$ is i.i.d. and has the same distribution as $\dil{1}{u}$. For a careful demonstration of the renewal argument, see for example the proof of \thref{cl:cond-renewal}, where additional conditioning is present.
	
	\begin{lemma} \thlabel{lem:finite-dist}
		For any $u\in \mathbb R$ and $i \in \mathbb N$, $\P(\dil{i}{u}< \infty) = \P(\dir{i}{u}< \infty) = q^i$.
	\end{lemma}
	
	\begin{proof}
		Let $\dil{i}{u} = X_1 + \cdots + X_i$ such that $X_1, \dots, X_i$ are i.i.d. random variables with the distribution of $\dil{1}{u}$. Then,
		\begin{align*}
			\P(\dil{i}{u} < \infty) = \left(\P(\dil{1}{u} < \infty)\right)^i = q^i.
		\end{align*}
		Similarly for $\P(\dir{i}{u}< \infty)$.
	\end{proof}
	
	For any $u$, define
	\begin{equation} \label{eq:cap-s-def}
		S := \sum_{i=1}^\infty \sum_{j=1}^\infty \beta^{i+j} \P\left(\dir{i}{u} < \dil{j}{u} < \infty\right).
	\end{equation}
	Note again that FCBA is translation invariant, so $S$ does not depend on $u$.
	
	\begin{lemma} \thlabel{lem:cap-s}
		$S = \frac12 \left(\frac{\beta q}{1 - \beta q}\right)^2$.
	\end{lemma}
	
	\begin{proof}
		We add two copies of $S$ with switched indices.
		\begin{align*}
			2S = \sum_{i=1}^\infty \sum_{j=1}^\infty \beta^{i+j} \left( \P(\dir{i}{u} < \dil{j}{u} < \infty) +  \P(\dir{j}{u} < \dil{i}{u} < \infty)\right).
		\end{align*}
		By symmetry, $\P(\dir{j}{u} < \dil{i}{u} < \infty) = \P(\dil{j}{u} < \dir{i}{u} < \infty)$. Then, 
		\begin{align*}
			\P(\dir{i}{u} < \dil{j}{u} < \infty) +  \P(\dil{j}{u} < \dir{i}{u} < \infty) = \P(\dir{i}{u} < \infty, \dil{j}{u} < \infty).
		\end{align*}
		By \thref{lem:finite-dist} and independence, the right-hand side is equal to $q^{i+j}$. Then, as claimed, 
		\begin{align*}
			2S = \sum_{i=1}^\infty \sum_{j=1}^\infty (\beta q)^{i+j} = \left(\frac{\beta q}{1-\beta q}\right)^2.
		\end{align*}
	\end{proof}

    \begin{lemma} \thlabel{lem:triple-dist}
        For any $i\geq 1$ and $u\in \mathbb R$,
        \begin{align*}
            \sum_{j=1}^\infty \beta^{i+j} \P(\dil{j-1}{u} < \dir{i}{u} < \dil{j}{u}< \infty) = \frac{(\beta q)^2}{2(1-\beta q)}.
        \end{align*}
    \end{lemma}

    \begin{proof}
        For any $u$, we observe that $\dil{j-1}{u} < \dil{j}{u}$ almost surely and the following equality holds
		\begin{align*}
			\left\{ \dil{j-1}{u} < \dir{i}{u} < \infty \right\} \cap \left\{ \dil{j}{u} < \infty \right\} = \left\{ \dil{j-1}{u} < \dir{i}{u} < \dil{j}{u} < \infty \right\} \sqcup \left\{\dil{j}{u} < \dir{i}{u} < \infty \right\}.
		\end{align*}
		Morever, $\dil{j}{u} = \dil{j-1}{u} + X$ where $X$ has the distribution of $\dil{1}{u'}$ for some $u'$. This along with independence allows us to write the probability of the event on the left-hand side as 
		\begin{align*}
			\P(\{\dil{j-1}{u} < \dir{i}{u} < \infty\} \cap \{\dil{j}{u} < \infty\}) = \P(\dil{j-1}{u} < \dir{i}{u} < \infty) \, \P(\dil{1}{u'} < \infty),
		\end{align*}
		where $\P(\dil{1}{u'} < \infty)= q$. By symmetry of the line, we have $\P(\dil{j-1}{u} < \dir{i}{u} < \infty) = \P(\dir{j-1}{u} < \dil{i}{u} < \infty)$ and $\P(\dil{j}{u} < \dir{i}{u} < \infty) = \P(\dir{j}{u} < \dil{i}{u} < \infty)$. Then, we have
		\begin{align} 
			&\sum_{i=1}^\infty \sum_{j=1}^\infty \beta^{i+j} \P(\dil{j-1}{x_1} < \dir{i}{x_1} < \dil{j}{x_1}< \infty) \label{eq:triple-dist-sum}\\
			&\qquad = q \sum_{i=1}^\infty \sum_{j=1}^\infty \beta^{i+j} \P(\dir{j-1}{u} < \dil{i}{u} < \infty) - \sum_{i=1}^\infty \sum_{j=1}^\infty \beta^{i+j} \P(\dir{j}{u} < \dil{i}{u} < \infty). \nonumber
		\end{align}
		Reindexing the first term on the right and we have
		\begin{align*}
			\eqref{eq:triple-dist-sum} = \beta q\sum_{i=1}^\infty \beta^i \P(\dir{0}{u} < \dil{i}{u}<\infty) + \beta q S - S.
		\end{align*}
		By \thref{lem:finite-dist}, 
		\begin{align*}
			\sum_{i=1}^\infty \beta^{i} \P(\dir{0}{u} < \dil{i}{u}<\infty) = \sum_{i=1}^\infty (\beta q)^{i} = \frac{\beta q}{1-\beta q}.
		\end{align*}
        Plugging in this and the formula for $S$ in \thref{lem:cap-s}, we have
        \begin{align*}
            \eqref{eq:triple-dist-sum} = \frac{(\beta q)^2}{2(1-\beta q)}.
        \end{align*}
    \end{proof}

\subsection{Two useful tools}
	\begin{proposition}
		Let $c = 1- (a+b)$ and $\xi = 1 - (\alpha + \beta)$. The following equations hold so long as the parameters in the denominators are nonzero:
		\begin{align}
			\frac{1}{a/2} \P((\R_1 \la \L)_{(0,\infty)}) 
			= \frac{1}{b} \P((\R_1 \lrlb \L)_{(0,\infty)}) 
			= \frac{1}{c} \P((\R_1 \lrl \L)_{(0,\infty)}) \label{eq:changeofmeasure-1}
			\\
			\P((u \overset{1}{\la} \L_k)_{(u,\infty)}) 
			= \frac{1}{\alpha} \P((u \overset{1}{\las} \L_k)_{(u,\infty)}) 
			= \frac{1}{\beta} \P((u \overset{1}{\law} \L_k)_{(u,\infty)})
			= \frac{1}{\xi} \P((u \overset{1}{\lrl} \L_k)_{(u, \infty)}) \label{eq:changeofmeasure-2}
		\end{align}
	\end{proposition}
	
	\begin{proof}
		These use the independence of the outcomes and their probabilities for each collision event.
	\end{proof}
 	\begin{remark} \thlabel{rem:c}
		As is in \cite[Remark 2.4]{benitez2023three}, if $c = 0$ and/or $\xi =0$, similar formulas as in \thref{lem:main-rec} could be derived using whichever parameters are nonzero.
	\end{remark}

	\begin{proposition}[Mass transport principle] 
		Consider a family of non-negative random variables Z(m,n) for integers $m,n \in \Z$ such that its distribution is diagonally invariant under translation, i.e., for any integer $\ell$, $Z(m+\ell,n+\ell)$ has the same distribution as $Z(m,n)$. Then for each $m \in \Z:$
		\begin{equation} \label{eq:mtp}
			\E \sum\limits_{n \in \Z} Z(m,n) = \E \sum\limits_{n \in \Z} Z(n,m).
		\end{equation}
	\end{proposition}
	\begin{proof}
		Using Fubini's theorem and translation invariance of $\E [Z(m,n)]$, we obtain
		\begin{align*}
			\E \sum\limits_{n \in \Z} Z(m,n) =  \sum\limits_{n \in \Z} \E [Z(m,n)] = \sum\limits_{n \in \Z} \E [Z(2m-n,m)]
			= \sum\limits_{n \in \Z} \E [Z(n,m)] = \E \sum\limits_{n \in \Z} Z(n,m).
		\end{align*}
	\end{proof}

\subsection{Recursion Quantities}
\begin{lemma} \thlabel{lem:main-rec}
		For FCBA restricted to $(0, \infty)$, it holds that
		\begin{align}
			\P((0 \la \b) \wedge \B_1)
			&= \frac{\alpha pq + \xi pq^2}{1 - \beta q} \label{eq:rec-2}.
		\end{align}
	\end{lemma}
	
	\begin{proof}
		We consider the two ways that $\B_1$ can be annihilated:
		\begin{align*}
			\P((0 \la \b) \wedge \B_1) = \P(\B_1 \las \L ) + \P((0 \la \b) \wedge (\B_1 \lrl \L)).
		\end{align*}
		For the first term, we break into cases by the number of weak collisions that occurred before $\B_1$ is destroyed from a strong collision. 
		\begin{align*}
			\P(\B_1 \las \L ) = \sum_{i=1}^\infty \P(\B_1 \wedge (x_1 \overset{i}{\las} \L)_{(x_1, \infty)}).
		\end{align*} 
		By \eqref{eq:changeofmeasure-2} and the renewal property we have
		\begin{align*}
			\P(\B_1 \las \L) = p \sum_{i=1}^\infty (\beta q)^{i-1} (\alpha q) = \frac{p\alpha q}{1-\beta q}.
		\end{align*} 
		For the second term, one again considers the number of weak collisions that occurred before $\B_1$ is mutually annihilated.
		\begin{align*}
			\P((0 \la \b) \wedge (\B_1 \lrl \L)) = \sum_{i=1}^\infty \P(\B_1 \wedge (x_1 \overset{i}{\lrl} \L)_{(x_1, x_1 + \dil{i}{x_1}]} \wedge (x_1 + \dil{i}{x_1} \la \L)_{(x_1 + \dil{i}{x_1}, \infty)}).
		\end{align*}
		Similarly, 
		\begin{align*}
			\P((0 \la \b) \wedge (\B_1 \lrl \L)) = p \sum_{i=1}^\infty (\beta q)^{i-1} (\xi q) q = \frac{p\xi q^2}{1-\beta q}.
		\end{align*}
		Combining both terms, we obtain the claimed formula.
	\end{proof}

	\begin{proposition} \thlabel{prop:sr} Let $s$ and $r$ be as defined at \eqref{eq:sr} and  $\hat{p} := \P((\R_1 \lrlb \L)_{[x_1, \infty)})$. It holds that
		\begin{align}
			s &= \frac{(p+\hat p)(1-\alpha)(1-\beta)q^2}{2 (1-\beta q)^2}
			\label{eq:s-final}\\
			r &= \frac{(p+\hat p) \alpha q(1-q)}{(1-\beta q)^2}. \label{eq:r-final}
		\end{align}
	\end{proposition}	
	
	In the derivations for both $s$ and $r$, we use the mass transport principle to write $s$ and $r$ in terms of the quantity $S$ and apply \thref{lem:cap-s,lem:triple-dist}.
	
	\begin{proof}[Proof of \eqref{eq:s-final}]
		Recall that 
		\begin{align*}
			s = \P((0\la \b) \wedge (\R_1 \la \B)) + \P((0\la \b) \wedge (\R_1 \la \BB)).
		\end{align*}
		Let $s = s_w + s_{mu}$ where
		\begin{align*}
			s_w &:= \P((0\la \b) \wedge (\R_1 \raw \B)) + \P((0\la \b) \wedge (\R_1 \raw \BB)) \\
			s_{mu} &:= \P((0\la \b) \wedge (\R_1 \lrl \B) + \P((0\la \b) \wedge (\R_1 \lrl \BB)). 
		\end{align*}
		
		For $s_w$ there are two cases, depending on how the blockade that annihilates $\R_1$ is annihilated. Either this bloackde is mutually anihilated from the right and another particle visits, or it is ``bulldozed'' from the right by a particle that visits the origin. In either cases, the annihilation of the blockade must occur after $\R_1$ weakly collides with that blockade, which introduces a comparison between distances. We define the following indicator:
		\begin{align*}
			Z_{w}(m,n) := \sum_{i=1}^\infty &\sum_{j=1}^\infty \ind{(\R_m \overset{i}{\raw} \B_n)_{[x_m, x_n]} \wedge (\B_n \overset{j}{\lrl} \L)_{[x_n, x_n + \dil{j}{x_n}]} \wedge (x_n + \dil{j}{x_n} \la \L)_{(x_n + \dil{j}{x_n},\infty)} \\
				& \qquad \qquad \wedge (\dir{i}{x_n} < \dil{j}{x_n} < \infty)} \\
				& \quad + \ind{(\R_m \overset{i}{\raw} \B_n)_{[x_m, x_n]} \wedge (\B_n \overset{j}{\las} \L)_{[x_n, \infty)} \wedge (\dir{i}{x_n} < \dil{j}{x_n} < \infty)} \\
				& \quad + \sum_{m<k<n} \ind{(\R_m \overset{i}{\raw} \BB_{k,n})_{[x_m, x_n]} \wedge (\BB_{k,n} \overset{j}{\lrl} \L)_{[x_k, x_n + \dil{j}{x_n}]} \\
				&\qquad \qquad \wedge (x_n + \dil{j}{x_n} \la \L)_{(x_n + \dil{j}{x_n},\infty)} \wedge (\dir{i}{x_k} < \dil{j}{x_n} < \infty)} \\
				&\quad + \sum_{m<k<n} \ind{(\R_m \overset{i}{\raw} \BB_{k,n})_{[x_m, x_n]} \wedge (\BB_{k,n} \overset{j}{\las} \L)_{[x_k, \infty)} \wedge (\dir{i}{x_k} < \dil{j}{x_n}< \infty)}.
		\end{align*}
        We note that the first three lines correspond to the two cases where $\R_m$ collides with an ``original'' blockade $\B$ and the last three lines correspond to the two cases where $\R_m$ collides with a generated blockade $\BB$. 
		By the mass transport principle \eqref{eq:mtp}, we have 
		\begin{align*}
			s_w = \E \sum_{n>1} Z_w(1,n) = \E \sum_{m<1} Z_w(m,1).
		\end{align*}
		We expand the right-hand side of the above:
		\begin{equation} \label{eq:zw-expand}
		\begin{split}
			\E \sum_{m<1} Z_w(m,1)
			= \sum_{i=1}^\infty &\sum_{j=1}^\infty \P\Big( \B_1 \wedge (\R_m \overset{i}{\raw} x_1)_{[x_m, x_1)} \wedge (x_1 \overset{j}{\lrl} \L)_{(x_1, x_1 + \dil{j}{x_1}]} \\
			&\qquad \qquad \wedge (x_1 + \dil{j}{x_1} \la \L)_{(x_1 + \dil{j}{x_1}, \infty)} \wedge (\dir{i}{x_1} < \dil{j}{x_1} < \infty)\Big) \\
			&\quad + \P\Big(\B_1 \wedge (\R_m \overset{i}{\raw} x_1)_{[x_m, x_1)} \wedge (x_1 \overset{j}{\las} \L)_{(x_1, \infty)} \wedge (\dir{i}{x_1} < \dil{j}{x_1} < \infty) \Big) \\
			&\quad + \sum_{m<k<1} \P \Big( (\R_k \lrlb \L_1)_{[x_k, x_1]} \wedge (\R_m \overset{i}{\raw} x_k)_{[x_m, x_k)} \wedge (x_1 \overset{j}{\lrl} \L)_{(x_1, x_1 + \dil{j}{x_1}]} \\
			&\qquad \qquad \wedge (x_1 + \dil{j}{x_1} \la \L)_{(x_1 + \dil{j}{x_1}, \infty)} \wedge (\dir{i}{x_k} < \dil{j}{x_1} < \infty)\Big) \\
			&\quad + \sum_{m<k<1} \P\Big((\R_k \lrlb \L_1)_{[x_k, x_1]} \wedge (\R_m \overset{i}{\raw} x_k)_{[x_m, x_k)} \wedge (x_1 \overset{j}{\las} \L)_{(x_1, \infty)} \\
			&\qquad \qquad \wedge (\dir{i}{x_k} < \dil{j}{x_1} < \infty) \Big).
		\end{split}
		\end{equation}
		
		Conditional on the event $\dir{i}{x_1} < \dil{j}{x_1} < \infty$, events occurring on disjoint intervals are independent. We carefully evaluate the following conditional probability as an example of \emph{the renewal argument} with conditioning and apply the same argument for the rest.

        \begin{claim} \thlabel{cl:cond-renewal}
            $\P((\R_m \overset{i}{\raw} x_1)_{[x_m, x_1)} \mid \dir{i}{x_1} < \dil{j}{x_1} < \infty) = \beta^i$.
        \end{claim}

        \begin{proof}
            We first partition the event by the indices of the particles that weakly visits $x_1$:
		\begin{align}
			&\P((\R_m \overset{i}{\raw} x_1)_{[x_m, x_1)} \mid \dir{i}{x_1} < \dil{j}{x_1} < \infty)  \label{eq:cond-example}\\
			&\qquad = \P\left(\bigcup_{m = m_i < m_{i-1} < \cdots < m_1 < m_0 =1} \bigcap_{\ell=1}^i \ (\R_{m_\ell} \overset{1}{\raw} x_{m_{\ell -1}})_{[x_{m_\ell}, x_{m_{\ell-1}})} \ \Bigg \vert \ \dir{i}{x_1} < \dil{j}{x_1} < \infty \right). \nonumber
		\end{align}
		Then, we apply the definition of conditional probability in a nested manner and obtain the following product of conditional probabilities:
		\begin{align*}
			\eqref{eq:cond-example} = \prod_{\ell=1}^i \P\left(\bigcup_{m_\ell < m_{\ell-1}}(\R_{m_\ell} \overset{1}{\raw} x_{m_{\ell-1}})_{[x_{m_\ell}, x_{m_{\ell-1}})} \ \Bigg \vert \ \bigcup_{m_{\ell-1}< \cdots < m_0 = 1} \bigcap_{\ell' =1}^{\ell-1} (\R_{m_{\ell'}} \overset{1}{\raw} x_{m_{\ell' -1}}), \dir{i}{x_1} < \dil{j}{x_1} < \infty\right). 
		\end{align*}
		By \eqref{eq:changeofmeasure-2}, 
		\begin{align*}
			&\P\left(\bigcup_{m_\ell < m_{\ell-1}}(\R_{m_\ell} \overset{1}{\raw} x_{m_{\ell-1}})_{[x_{m_\ell}, x_{m_{\ell-1}})} \ \Bigg \vert \ \bigcup_{m_{\ell-1}< \cdots < m_0 = 1} \bigcap_{\ell' =1}^{\ell-1} (\R_{m_{\ell'}} \overset{1}{\raw} x_{m_{\ell' -1}}), \dir{i}{x_1} < \dil{j}{x_1} < \infty\right) \\
			&\quad = \beta \P\left(\bigcup_{m_\ell < m_{\ell-1}}(\R_{m_\ell} \overset{1}{\ra} x_{m_{\ell-1}})_{[x_{m_\ell}, x_{m_{\ell-1}})} \ \Bigg \vert \ \bigcup_{m_{\ell-1}< \cdots < m_0 = 1} \bigcap_{\ell' =1}^{\ell-1} (\R_{m_{\ell'}} \overset{1}{\raw} x_{m_{\ell' -1}}), \dir{i}{x_1} < \dil{j}{x_1} < \infty\right) \\
			&\quad = \beta \P\left((\R \ra x_{m_{\ell-1}})_{(\infty, x_{m_{\ell-1}})} \ \Bigg \vert \ \bigcup_{m_{\ell-1}< \cdots < m_0 = 1} \bigcap_{\ell' =1}^{\ell-1} (\R_{m_{\ell'}} \overset{1}{\raw} x_{m_{\ell' -1}}), \dir{i}{x_1} < \dil{j}{x_1} < \infty\right).
		\end{align*}
		Note that the probability in the last line is $1$ since we are conditioning on $\dir{i}{x_1}< \infty$. Thus,
		\begin{align*}
			\P((\R_m \overset{i}{\raw} x_1)_{[x_m, x_1)} \mid \dir{i}{x_1} < \dil{j}{x_1} < \infty)  = \beta^i.
		\end{align*}
        \end{proof}

		Applying the above argument to the conditional events on each disjoint interval and we have that the first term on the right-hand side of \eqref{eq:zw-expand} is equal to
		\begin{align*}
			p \xi \beta^{-1} q \beta^{i+j} \P(\dir{i}{x_1} < \dil{j}{x_1} < \infty).
		\end{align*}
		The second term on the right-hand side of \eqref{eq:zw-expand} is equal to
		\begin{align*}
			p \alpha\beta^{-1} \beta^{i+j} \P(\dir{i}{x_1} < \dil{j}{x_1} < \infty).
		\end{align*}
		
		Note that by \eqref{eq:changeofmeasure-1}, $\hat p = (b/c)\P(\R_1 \lrl \L)$. For the third and fourth terms on the right-hand side of \eqref{eq:zw-expand}, we note that since $x_k < x_1$, conditional on $\dil{j}{x_1}$, $\dir{i}{x_k}$ and $\dir{i}{x_1}$ have the same distribution for any $i$ and $j$. Thus 
		\begin{align*}
			\P(\dir{i}{x_k} < \dil{j}{x_1} < \infty)= \P(\dir{i}{x_1} < \dil{j}{x_1} < \infty).
		\end{align*}
		Then, the third and fourth terms on the right-hand side of \eqref{eq:zw-expand} after summing over $k$ are equal to
		\begin{align*}
			\hat{p} \xi \beta^{-1} q \beta^{i+j} \P(\dir{i}{x_1} < \dil{j}{x_1} < \infty) \quad \text{and} \quad \hat{p} \alpha \beta^{-1} \beta^{i+j} \P(\dir{i}{x_1} < \dil{j}{x_1} < \infty).
		\end{align*}
		
		Summing all four terms over $i$ and $j$ and we have
		\begin{align*}
			\E \sum_{m<1} Z_w(m,1) = (p + \hat{p}) \beta^{-1}(\xi q + \alpha)  \sum_{i=1}^\infty \sum_{j=1}^\infty \beta^{i+j} \P(\dir{i}{x_1} < \dil{j}{x_1} < \infty) = (p + \hat{p}) \beta^{-1}(\xi q + \alpha)  S.
		\end{align*}
		By \thref{lem:cap-s}, we have 
		\begin{equation} \label{eq:sw-final}
			s_w = \frac{(p+\hat{p})(\beta q^2) (\xi q + \alpha)}{2(1-\beta q)^2}.
		\end{equation}
		
		Next let us consider 
		\begin{align*}
			s_{mu} = \P((0\la \b) \wedge (\R_1 \lrl \B) + \P((0\la \b) \wedge (\R_1 \lrl \BB)).
		\end{align*}
		On these events, after $\R_1$ mutually annihilates with a blockade, another left moving particle must visit the location of the blockade from the right and then visit the origin. The ordering of the visits gives rise to an event comparing three times: the time of the last weak collision with the blockade from the right, the time $\R_1$ mutually annihilates with the blockade, and the time a $j$th particle visits the location of the blockade from the right. We define the following indicator:
		\begin{align*}
			Z_{mu} (m,n) 
			:= \sum_{i=1}^\infty &\sum_{j=1}^\infty \ind{(\R_m \overset{i}{\lrl} \B_n)_{[x_m, x_n]} \wedge (\B_n \overset{j-1}{\law} \L)_{[x_n, x_n + \dil{j-1}{x_n}]} \wedge (x_n + \dil{j-1}{x_n} \la \L)_{(x_n + \dil{j-1}{x_n}, \infty)} \\
			&\qquad \qquad \wedge (\dil{j-1}{x_n} < \dir{i}{x_n} < \dil{j}{x_n}< \infty)} \\
			&\quad + \sum_{m<k<n} \ind{(\R_m \overset{i}{\lrl} \BB_{k,n})_{[x_m, x_n]} \wedge (\BB_{k,n} \overset{j-1}{\law} \L)_{[x_k, x_n + \dil{j-1}{x_n}]} \\
			&\qquad \qquad \wedge (x_n + \dil{j-1}{x_n} \la \L)_{(x_n + \dil{j-1}{x_n}, \infty)} \wedge (\dil{j-1}{x_n} < \dir{i}{x_k} < \dil{j}{x_n}< \infty)}
		\end{align*}
		Again, by the mass transport principle \eqref{eq:mtp}, we have 
		\begin{align*}
			s_{mu} = \E \sum_{n>1} Z_{mu}(1,n) = \E \sum_{m<1} Z_{mu}(m,1).
		\end{align*}
		With essentially the same arguments for \eqref{eq:zw-expand}, we have
		\begin{align*}
			\E \sum_{m<1} Z_{mu} (m,1) = (p+ \hat{p})\xi\beta^{-2} \sum_{i=1}^\infty \sum_{j=1}^\infty \beta^{i+j} \P(\dil{j-1}{x_1} < \dir{i}{x_1} < \dil{j}{x_1}< \infty).
		\end{align*}
		By \thref{lem:triple-dist}, we have
		\begin{equation} \label{eq:smu-final}
			s_{mu} =  \frac{(p+\hat p)\xi q^2 }{2 (1-\beta q)}.
		\end{equation}
		
        Combining \eqref{eq:sw-final} and \eqref{eq:smu-final} we have
		\begin{equation}
			s = s_w + s_{mu} = \frac{(p+\hat p)(1-\alpha)(1-\beta)q^2}{2 (1-\beta q)^2}.
		\end{equation}
	\end{proof}

	\begin{proof}[Proof of \eqref{eq:r-final}]
		Recall that
		\begin{align*}
			r = \P((0 \not \la \b) \wedge (\R_1 \la \B)) + \P((0 \not \la \b) \wedge (\R_1 \la \BB)).
		\end{align*}
		As is in the derivation of $s$, we let $r = r_w + r_{mu}$ where
		\begin{align*}
			r_w &:=  \P((0 \not \la \b) \wedge (\R_1 \raw \B)) + \P((0 \not \la \b) \wedge (\R_1 \raw \BB)\\
			r_{mu} &:= \P((0 \not \la \b) \wedge (\R_1 \lrl \B) + \P((0 \not \la \b) \wedge (\R_1 \lrl \BB)).
		\end{align*}
		
		There are two cases for $r_w$. After $\R_1$ weakly collides with a blockade, either all following collisions with the blockades are weak, or the blockade is mutually annihilated from the right and no other particle visits from the right. For the second case to happen, the blockade needs to annihilate $\R_1$ before it mutually annihilates with a particle on the right. We define the following indicator:
		\begin{align*}
			Y_w(m,n) := \sum_{i=1}^\infty &\sum_{j=1}^\infty \ind{(\R_m \overset{i}{\raw} \B_n)_{[x_m, x_n]} \wedge (\B_n \overset{j-1}{\law} \L)_{[x_n, \infty)}} \\
			&\quad + \ind{(\R_m \overset{i}{\raw} \B_n)_{[x_m, x_n]} \wedge (\B_n \overset{j}{\lrl} \L)_{[x_n, x_n + \dil{j}{x_n}]} \wedge (x_n + \dil{j}{x_n} \not \la \L)_{(x_n + \dil{j}{x_n}, \infty)} \\
			&\qquad \qquad \wedge (\dir{i}{x_n} < \dil{j}{x_n}< \infty)} \\
			&\quad + \sum_{m<k<n} \ind{(\R_m \overset{i}{\raw} \BB_{k,n})_{[x_m, x_n]} \wedge (\BB_{k,n} \overset{j-1}{\law} \L)_{[x_k, \infty)}} \\
			&\quad + \sum_{m<k<n} \ind{(\R_m \overset{i}{\raw} \BB_{k,n})_{[x_m, x_n]} \wedge (\BB_{k,n} \overset{j}{\lrl} \L)_{[x_k, x_n + \dil{j}{x_n}]} \\
			&\qquad \qquad \wedge (x_n + \dil{j}{x_n} \not \la \L)_{(x_n + \dil{j}{x_n}, \infty)}\wedge (\dir{i}{x_k} < \dil{j}{x_n}<\infty)}.
		\end{align*}
		By the mass transport principle \eqref{eq:mtp}, we have 
		\begin{align*}
			r_w = \E \sum_{n>1} Y_w(1,n) = \E \sum_{m<1} Y_w(m,1).
		\end{align*}
		We expand the right-hand side of the above:
		\begin{align}
			\E \sum_{m<1} Y_w(m,1)
			= \sum_{i=1}^\infty &\sum_{j=1}^\infty \P\Big(\B_1 \wedge (\R_m \overset{i}{\raw} x_1)_{[x_m, x_1)} \wedge (x_1 \overset{j-1}{\law} \L)_{(x_1, \infty)} \Big) \label{eq:yw-1}\\
			&\quad + \P \Big(\B_1 \wedge (\R_m \overset{i}{\raw} x_1)_{[x_m, x_1)} \wedge (x_1 \overset{j}{\lrl} \L)_{(x_1, x_1 + \dil{j}{x_1}]} \\
			&\qquad \qquad \wedge (x_1 + \dil{j}{x_1} \not \la \L)_{(x_1 + \dil{j}{x_1}, \infty)} \wedge (\dir{i}{x_1}< \dil{j}{x_1} < \infty)\Big) \label{eq:yw-2}\\
			&\quad + \sum_{m<k<1} \P \Big( (\R_k \lrlb \L_1)_{[x_k, x_1]} \wedge (\R_m \overset{i}{\raw} x_k)_{[x_m, x_k)} \wedge (x_1 \overset{j-1}{\law} \L)_{(x_1, \infty)}\Big) \label{eq:yw-3}\\
			&\quad + \sum_{m<k<1} \P \Big( (\R_k \lrlb \L_1)_{[x_k, x_1]} \wedge (\R_m \overset{i}{\raw} x_k)_{[x_m, x_k)} \wedge (x_1 \overset{j}{\lrl} \L)_{(x_1, x_1 + \dil{j}{x_1}]} \\
			&\qquad \qquad \wedge (x_1 + \dil{j}{x_1} \not \la \L)_{(x_1, \dil{j}{x_1}, \infty)} \wedge (\dir{i}{x_k} < \dil{j}{x_1} < \infty)\Big). \label{eq:yw-4}
		\end{align}
		For the terms \eqref{eq:yw-1} and \eqref{eq:yw-3}, we note that $(x_1 \overset{j-1}{\law} \L)_{(x_1, \infty)}$ implies that there is no $j$-th particle that visits $x_1$ from the right. The renewal argument then implies
		\begin{align*}
			\eqref{eq:yw-1} + \eqref{eq:yw-3} = \sum_{i=1}^\infty \sum_{j=1}^\infty (p+ \hat p) (\beta q)^{i+ j-1} (1-q) = \frac{(p+\hat p)(1-q)\beta q}{(1-\beta q)^2}.
		\end{align*}
		The terms \eqref{eq:yw-2} and \eqref{eq:yw-4} follow from similar arguments used in the proof for $s$ and 
		\begin{align*}
			\eqref{eq:yw-2} + \eqref{eq:yw-4} = (p + \hat p)  \xi  (1-q) \sum_{i=1}^\infty \sum_{j=1}^\infty\beta^{i+j-1} \P(\dir{i}{x_1} < \dil{j}{x_1} < \infty) = (p + \hat p)\xi \beta^{-1} (1-q) S.
		\end{align*}
		Together, 
		\begin{equation} \label{eq:rw-final}
			r_w = \frac{(p+ \hat p)\beta q(1-q)(2+\xi q)}{2(1-\beta q)^2}.
		\end{equation}
		
		For $r_{mu}$, the location of the blockade that mutually annihilates with $\R_1$ must not be visited from the right after the mutual annihilation occurs. We define the following indicator:
		\begin{align*}
			Y_{mu}(m,n) := \sum_{i=1}^\infty &\sum_{j=1}^\infty \ind{(\R_m \overset{i}{\lrl} \B_n)_{[x_m, x_n]} \wedge (\B_n \overset{j-1}{\law} \L)_{[x_n, \infty)} \wedge (\dil{j-1}{x_n}< \dir{i}{x_n} <\infty)} \\
			&\quad + \sum_{m<k<n} \ind{(\R_m \overset{i}{\lrl} \BB_{k,n})_{[x_m, x_n]} \wedge (\BB_{k,n} \overset{j-1}{\law} \L)_{[x_k, \infty)} \wedge (\dil{j-1}{x_n}< \dir{i}{x_k} <\infty)}.
		\end{align*}
		Again, by the masstransport principle \eqref{eq:mtp}, we have
		\begin{align*}
			r_{mu} = \E \sum_{n>1} Y_{mu}(1,n) = \E \sum_{m<1} Y_{mu}(m,1).
		\end{align*}
		By similar arguments as before, we have
		\begin{align*}
			\E \sum_{m<1} Y_{mu}(m,1) &= (p+\hat p) \xi \beta^{-1} (1-q) \sum_{i=1}^\infty \sum_{j=1}^\infty \beta^{i+j-1} \P(\dil{j-1}{x_1}< \dir{i}{x_1} <\infty) \\
			&= (p+\hat p) \xi \beta^{-1} (1-q) \left(\frac{\beta q}{1-\beta q}+ S \right).
		\end{align*}
		By \thref{lem:cap-s}, we have
		\begin{equation} \label{eq:rmu-final}
			r_{mu} = \frac{(p+ \hat p) \xi q (1-q) (2-\beta q)}{2(1-\beta q)^2}.
		\end{equation}
		
        Combining \eqref{eq:rw-final} and \eqref{eq:rmu-final} we have
		\begin{align*}
			r = r_w + r_{mu} = \frac{(p+\hat p) \alpha q(1-q)}{(1-\beta q)^2}.
		\end{align*}
	\end{proof}

	\bibliographystyle{alpha}
	\bibliography{BA}
	
\end{document}